\documentclass[12pt]{article}

\usepackage{amsfonts}
\usepackage{amsmath, amsthm}
\usepackage{amssymb}
\usepackage{indentfirst}

\usepackage[russian, english]{babel}
\usepackage{latexsym}
\renewcommand{\le}{\leqslant}
\renewcommand{\ge}{\geqslant}

\newcommand{\R}{\mathbb R}

\begin{document}
\title{Acute sets}
\author{Dmitriy Zakharov \thanks{Moscow, school 179}}
\date{}
\maketitle
\newtheorem{theorem}{Theorem}
\newtheorem{lemma}{Lemma}
\newtheorem{claim}{Proposition}
\renewcommand{\refname}{References}

\begin{abstract} A set of points in $\R^d$ is {\it acute} if any three points from this set form an acute triangle. In this note we construct an acute set in $\R^d$ of size at least $1.618^d$. We also present a simple example of an acute set of size at least $2^{\frac{d}{2}}$. \end{abstract}

\section*{Introduction}

A set of points in $\R^d$ is {\it acute}, if any three points of this set form an acute triangle.
In 1962 Danzer and Gr\"{u}nbaum \cite{DG} posed the following question: what is the maximum size $f(d)$ of an acute set in $\R^d$? They proved a linear lower bound $f(d) \ge 2d-1$ and conjectured that this bound is tight. However, in 1983 Erd\H os and F\"uredi \cite{EF} disproved this conjecture in large dimensions. They gave an exponential lower bound
\begin{equation} \label{ef}
f(d) \ge \frac{1}{2} \left ( \frac{2}{\sqrt{3}} \right )^d > 0.5 \cdot 1.154^d.
\end{equation}
Their proof is a very elegant application of the probabilistic method. One drawback of their approach is that only the existence of an acute set of such size is proven, with no possibility to turn it into an explicit construction. 

In 2009 Ackerman and Ben-Zwi \cite{AZ} improved $(\ref{ef})$ by factor $\sqrt{d}$:
$$
f(d) \ge c\sqrt{d} \left ( \frac{2}{\sqrt{3}} \right )^d 
$$

In 2011 Harangi \cite{H} refined the approach of Erd\H os and F\"uredi and improved their bound to
$$
f(d) \ge c \left ( \sqrt[10]{\frac{144}{23}} \right )^d>c \cdot 1.2^d.
$$

In this note we give a simple proof of the following inequality:

\begin{theorem}\label{thm2}
$f(d+2)\ge 2f(d) $, that is, for any $d$-dimensional acute set there exists a $(d+2)$-dimensional acute set of twice the size.
\end{theorem}
Theorem~\ref{thm2} implies lower bound for $f(d)$:
$$
f(d) \ge 2^{\frac{d}{2}}
$$ 

Let $F_d$ be the $d$-th Fibonacci number, that is $F_0 = F_1 = 1$ and $F_{d+2} = F_{d+1}+F_d$. Also, we prove the following inequality:

\begin{theorem}\label{thm1}
There exist $d$-dimensional acute sets of size $F_{d+1}$ that is, $f(d) \ge F_{d+1}$.
\end{theorem}

Using the formula for Fibonacci numbers we can write an asymptotic inequality for $f(d)$:
$$
f(d) \ge \left ( \frac{1+\sqrt{5}}{2} \right )^{d} \ge 1.618^{d}
$$

Proofs of Theorem~\ref{thm2} and \ref{thm1} are explicit and allow to construct acute sets effectively.\\

The best known upper bound on $f(d)$ is $f(d)\le 2^d-1$, and follows from the main result of \cite{DG}. Danzer and Gr\"unbaum proved that if a set $S$ of points in $\R^d$ determines only acute and right angles, then $|S|\le 2^d$. Moreover, if $|S|=2^d$ then $S$ must be an affine image of a $d$-dimensional cube.
\\

\section*{Proof of Theorem~\ref{thm2}}

Proofs of both theorems are based on two simple propositions:

\begin{claim}\label{pr1}
For any points $a, b, c$, forming an acute angle, there is $\epsilon > 0$ such that for all $\tilde a, \tilde b, \tilde c$,  $||a-\tilde a||, ||b-\tilde b||, ||c-\tilde c|| < \epsilon$, angle $(\tilde a, \tilde b, \tilde c)$ is acute too.
\end{claim}

\begin{claim}[The key fact]\label{pr2}
Suppose that $X \subset \R^d$ is an acute set and $r>0$ is a sufficiently small number. For each $x \in X$ we take an arbitary point $\phi(x) \in \R^2$ on the circle of radius $r$ with center in the origin such that all points $\pm \phi(x)$ are different. Then the set $Y = \{(x, \pm \phi(x))|x\in X\} \subset \R^{d+2}$ is acute as well.
\end{claim}

To prove Theorem \ref{thm2}, we apply Proposition \ref{pr2} to a maximal acute set $X$ in $\R^d$, $|X|=f(d)$. We get an acute set $Y \subset \R^{d+2}$  of size $|Y|=2|X|$ which proves the theorem.

\begin{proof}[Proof of Proposition 2] 

 The scalar product of two vectors $u,v$ is denoted by $\langle u,v\rangle$. Put $$s:=\min\{\langle y-x,z-x\rangle: x, y, z \in X, x \neq y, x \neq z\}.$$ Since the set $X$ is acute, we have $s>0$, and we can take a positive number $r$ such that $4r^2 < s$.

 Our aim is to prove that $Y$ is acute. Take three distinct points $\tilde x,\tilde y, \tilde z \in Y$, where
$$
\tilde x = (x, a\phi(x)),\ \ \ \tilde y = (y, b\phi(y)),\ \ \ \tilde z = (z, c\phi(z)),\ \ \ a,b,c\in\{\pm 1\}.
$$

Suppose that $x \neq y$ and $x \neq z$. Then
\begin{equation} \label{1}
\langle\tilde y - \tilde x,\tilde z - \tilde x\rangle=\langle y-x,z-x\rangle+\langle b\phi(y)-a\phi(x),c\phi(z)-a\phi(x)\rangle.
\end{equation}
The first scalar product on the right hand side is at least $s$ by the definition of $s$, while the second scalar product is at most $4r^2$. By the choice of $r$, the sum of these two scalar products is positive, which means that the angle $(\tilde y,\tilde x, \tilde z)$ is acute.

Suppose that $x=y$ (the case $x=z$ is treated in the same way). We have $a+b=0$, so
$$
\langle\tilde y - \tilde x,\tilde z - \tilde x\rangle=\langle b\phi(y)-a\phi(x),c\phi(z)-a\phi(x)\rangle=\langle 2a\phi(x), a\phi(x)-c\phi(z)\rangle=2\big(\|\phi(x)\|^2\pm \langle\phi(x),\phi(z)\rangle\big) >0,
$$
because $\phi(x) \neq \pm\phi(z)$. Thus, the angle $(\tilde y,\tilde x, \tilde z)$ is acute in this case as well.

Proposition $\ref{pr2}$ is proved.

\end{proof}

\section*{Proof of Theorem~\ref{thm1}}

\textbf{Sketch of the proof.} We prove that there exists a $d$-dimensional acute set of size $F_{d+1}$ with the property that there exists a $(d-1)$-dimensional hyperplane $H \subset \R^d$ such that $H$ contains $F_d$ points and the remaining $F_{d-1}$ points are on the same side of $H$.

The proof is by induction. The basic idea is the same as in the first construction: we want to replace a point $v$ with two points $v \pm \phi(v)$. However, this time we have only one extra dimension. So we do this only for the points on the hyperplane $H$. It is not hard to see that if we choose the vectors $\phi(v)$ carefully, then this results in a $(d + 1)$-dimensional acute set of size $F_{d+2}$. (To get the ``hyperplane property", one needs to modify this set a bit. So some extra work needs to be done here but it is mainly technical.)
\vskip 0.5cm

We will derive by induction Theorem \ref{thm1} from the following lemma:

\begin{lemma}
Suppose that $X \subset \R^d$ is an acute set and  $h$ is a hyperplane such that $X$ lies on one side of $h$. Then there is an acute set $X' \subset \R^{d+1}$ and a hyperplane $H$ in $\R^{d+1}$ such that $|X'|= |X|+|X\cap h|$, $|X' \cap H| = |X|$ and $X'$ lies on one side of $H$.
\end{lemma}

\begin{proof}[Proof of Theorem~\ref{thm1}]
For $d=1$ we take $X=\{0, 1\} \subset \R^1$ and hyperplane $h = \{x~\in~\R^1|x~=~0\}$. Clearly, $|X| = F_2, |X \cap h| = F_1$ and the pair $(X, h)$ satisfies Lemma conditions.

Suppose that we constructed an acute set $X \subset \R^d$ and a hyperplane $h$ such that $|X|=F_{d+1}, |X \cap h| = F_d$ and $X$ lies on one side of $h$. Then, by Lemma 1, there is an acute set $X'\subset \R^{d+1}$ and a hyperplane $H$ such that 
$$
|X'|=|X|+|X\cap h|=F_{d+1}+F_d=F_{d+2}, ~ |X' \cap H| = |X| = F_{d+1}
$$ 
and $X'$ lies on one side of $H$. So  the induction step is completed.

\end{proof}

The proof of the Lemma is based on propositions \ref{pr1} and \ref{pr2}.

\begin{proof}[Proof of Lemma 1] We can assume that 
$$
h = \{(x_1, \ldots, x_{d-1}, 0, 0) | \,x_i \in \R\}
$$
$$
X \subset P = \{(x_1, \ldots, x_{d}, 0)  | \,x_i \in \R\}
$$
Let $A = X \cap h$, $B = X \setminus A$. 

Consider a $(d-1)$-plane $h_2 \subset P$ parallel to $h$ such that $X$ lies between $h$ and $h_2$. Let 
$$
h_3 = h + (0, \ldots, 0, r) \subset \R^{d+1},
$$
 where $r > 0$ is a sufficiently small positive number. Let $H \subset \R^{d+1}$ be the hyperplane passing through $h_2$ and $h_3$. Consider sets $A_+ = A + (0, \ldots, 0, 0, r),\, A_- = A - (0, \ldots, 0, 0, r)$ and let $B_H$ be the orthogonal projection of $B$ onto $H$.
 
\begin{claim}\label{pr3}
For a sufficiently small $r$ and arbitary $x, y, z \in A_+ \cup A_- \cup B_H$ such, that $\{x, y, z\} \not \subset A_+ \cup A_-$, triangle $\{x, y, z\}$ is acute. 
\end{claim}

\begin{proof}[Proof of Proposition 3] The distance between $x \in X$ and any corresponding point $\tilde x \in A_+ \cup A_- \cup B_H$ is at most $r$, so for all sufficiently small $r$ an obtuse angle can occur only in triangles $\{x, y, z\} = \{ a_+, a_-, b\}$, where $a_+=(a, 0, r), a_-=(a, 0, -r)$ and $(a, 0, 0) \in A, b \in B_H$. Since the distance between $a_+$ and $a_-$ equals to $2r$, the distances between $a_{\pm}$ and $b$ are bounded from below by a number not depending on $r$, therefore, angle $b$ is acute for small $r$. By a choice of $H$, point $b$ lies between hyperplanes $P \pm (0, \ldots, 0, r)$, thus angles $(a, 0, r)$ and $(a, 0, -r)$ of triangle $\{ a_+, a_-, b\}$ are acute too. 

\end{proof}

For each $x \in A$ we denote by $C(v) \subset \R^{d+1}$ the circle of radius $r$ with center in $v$ and orthogonal to $h$. 

\begin{claim}\label{pr4}
For any $\epsilon > 0$ there is a hyperplane $H_2$ such that:

 1. For each point $v \in B_H$ distance from $v$ to $H_2$ is less than $\epsilon$. 
 
 2. For each point $(v, 0, 0) \in A$ there exists a point $ \bar v = (v,\,\, \phi(v)) \in H_2 \cap C(v)$ such that $||(v, 0, r) -\bar v|| < \epsilon$ and all points $\pm \phi(v)$ are distinct.
 
\end{claim}

\begin{proof}[Proof of Proposition 4] Let $u \in \R^{d+1}:$  $H = \{ x \in \R^{d+1}| \langle x, u \rangle = 1\}$. Take a vector $\alpha \in \R^{d+1}$ such that $||\alpha||<\delta$ for sufficiently small $\delta > 0$ and $\alpha$ is not orthogonal to any of the vectors $v_1 - v_2$ where $v_1, v_2 \in A$. We take $H_2 = \{ x \in \R^{d+1}| \langle x, u+\alpha \rangle = 1\}$. 

1. For $v \in B_H$
$$
\rho(v, H_2) = \frac{|\langle v, u+\alpha \rangle - 1|}{||u+\alpha||}\le \frac{|\langle v, \alpha \rangle|}{||u||-\delta} \le \delta \frac{||v||}{||u||-\delta} < \epsilon
$$
as $\delta$ is sufficiently small.

2. Consider the intersection $l$ and $l_2$ of the hyperplanes $H$ and $H_2$ with the $2$-dimensional plane $\{(v, x_d, x_{d+1})|x_i \in \R \}$, where $(v, 0, 0) \in A$. Clearly, $l$ intersects $C(v)$ in two points (one of them is $(v, 0, r)$), and so for small $\delta$ line $l_2$ intersects $C(x)$ in two points as well, and also one of these points tends to $(v, 0, r)$ as $\delta \to 0$. This point we denote by $(v, \phi(v))$. It is sufficient to show that all points $\pm \phi(v)$ are distinct for $(v, 0, 0) \in A$.

As $||\phi(v)-(0, r)|| < r$ for all $\delta < r$,  thus $\phi(v_1) \not = -\phi(v_2)$. Take $\bar v_1 = (v_1, 0, 0), \bar v_2=(v_2, 0, 0) \in A$. If $\phi(v_1)=\phi(v_2)$, then 
$$
(v_1, \phi(v_1)) - (v_1, 0, 0) = (v_2, \phi(v_2)) - (v_2, 0, 0),
$$
that is
$$
(v_1, \phi(v_1))-(v_2, \phi(v_2))=\bar v_1 - \bar v_2 = \bar w
$$
but $(v_1, \phi(v_1))$ and $(v_2, \phi(v_2))$ lie in $H_2$, consequently $\bar w$ is orthogonal to $u + \alpha$ which contradicts the defenition of $\alpha$. Therefore all points $\pm \phi(v)$ are distinct.

\end{proof}

Now we take sufficiently small $\epsilon$ and corresponding hyperplane $H_2$ and a map $\phi$. Let $\tilde B$ be the orthogonal projection of $B_H$ onto $H_2$, also let
$$
\tilde A_+=\{ (x, \phi(x)) | (x, 0, 0) \in A\},\,\, \tilde A_-=\{ (x, -\phi(x)) | (x, 0, 0) \in A\}
$$

Combining Propositions 1 and 4 we can claim that the Proposition 3 is still true for corresponding sets $\tilde A_+, \tilde A_-$ and $ \tilde B$. We have to check that the set $Y = \tilde A_+ \cup \tilde A_-$ is acute. But it immediately follows from Proposition 2. We conclude that the set $X' = \tilde A_+ \cup \tilde A_- \cup \tilde B$ is acute. We have $|X'|=|\tilde A_+|+|\tilde A_-|+|\tilde B|=|X|+|X \cap h|$, also $\tilde A_+ \cup \tilde B \subset H_2$ and $X'$ lies on one side of $H_2$. Thus, the pair $(X', H_2)$  satisfies the conditions of Lemma 1.
\end{proof}

 \textsc{Acknowledgements}: We would like to thank Andrey Kupavskii and Alexandr Polyanskii for discussions that helped to improve the main result, as well as for their help in preparing this note. We would also like to thank Prof. Raigorodskii for introducing us to this problem and for his constant encouragement.

\end{document}